\documentclass[12pt]{amsart}
\usepackage{amssymb}
\usepackage[latin1]{inputenc}
\usepackage{amsthm}
\usepackage{amsfonts}
\usepackage{mathrsfs}
\usepackage{graphicx}
\usepackage{amsmath}
\usepackage[all]{xy} 
\numberwithin{equation}{section}

\newtheorem{teo}{Theorem}[section] 
\newtheorem{lem}[teo]{Lemma}

\newtheorem{defi}[teo]{Definition}
\newtheorem{oss}[teo]{Remark}

\newtheorem{prop}[teo]{Proposition} 
\newtheorem{conjecture}[teo]{Conjecture}

\title{One remark on the b-semiampleness of the moduli part}
\date{\today}

\author{Enrica Floris}

\address{Enrica Floris\\
Department of Mathematics, Imperial College London, 180 Queen's
Gate, London SW7 2AZ, UK}
\email{e.floris@imperial.ac.uk}

\begin{document}
\begin{abstract}
In this short note we reduce the b-semiampleness conjecture for lc-trivial fibrations to
the b-semiampleness conjecture for klt-trivial fibrations.
\end{abstract}

\maketitle

\section{Introduction}
An important problem in birational geometry is the following:
given a morphism $f\colon X\rightarrow Z$ and an lc pair $(X,B)$
such that $K_X+B\sim_{\mathbb{Q}f}0$, find a $\mathbb{Q}$-divisor
$\Delta_Z$ on $Z$ such that $(Z,\Delta_Z)$ is lc and
$$K_X+B\sim_{\mathbb{Q}}f^{\ast}(K_Z+\Delta_Z).$$
The problem has been solved by Ambro when $(X,B)$ is klt
and $B\geq 0$ by using the canonical bundle formula ~\cite[Theorem 0.2]{Amb05}
and by Fujino and Gongyo when $f$ is generically finite and $(X,B)$ is lc by ~\cite[Lemma 1.1]{FGsubba}.\\
In general, we write canonical bundle formulas for lc-trivial (resp. klt-trivial
fibrations), that are fibrations $f\colon(X,B)\rightarrow Z$
such that $(X,B)$ is lc (resp. klt) on the generic point of $Z$
and $K_X+B\sim_{\mathbb{Q}f}0$ (see Definition \ref{lctrivialfibr} for a precise definition).
The canonical bundle formula consists in writing 
$$K_X+B\sim_{\mathbb{Q}}f^{\ast}(K_Z+B_Z+M_Z)$$
where $B_Z$ is called \textit{discriminant} and contains information on the singular fibers of $f$
and $M_Z$ is called the \textit{moduli part} and, conjecturally,
is related to the birational variation of the fibers.
The main steps in the proof of ~\cite[Theorem 0.2]{Amb05} are the following:
\begin{itemize}
\item the pair $(Z,B_Z)$ has the same singularities as $(X,B)$ ~\cite[Proposition 3.4]{Amb99}, namely
\begin{center}
 $(Z,B_Z)$ is klt (resp. lc) $\Leftrightarrow$ $(X,B)$ is klt (resp. lc);
\end{center}
\item if $(X,B)$ is klt on the generic point of $Z$, then $M_Z$
is the pullback of a big divisor ~\cite[Theorem 3.3]{Amb05} or, more precisely,
there exist a generically finite morphism $\tau\colon\bar{Z}\rightarrow Z$,
a morphism $\rho\bar{Z}\rightarrow \colon Z^{!}$ and a big divisor $M_{Z^{!}}$
such that $$\tau^{\ast}M_Z=\rho^{\ast}M_{Z^{!}}.$$
\end{itemize}
Since klt is an open condition and the pair $(Z,B_Z)$ is klt,
it is sufficient to chose
$$\Delta_Z=B_Z+\frac{1}{m}\tau_{\ast}\rho^{\ast}D$$
where $D\in |mM_{Z^{!}}|$ and $m$ is big enough.

If we want to follow the path of this proof, in order to solve the problem for
$(X,B)$ is lc, we need a stronger condition on $M_Z$,
namely
\begin{conjecture}\label{bSemiLC}
Let $f\colon(X,B)\rightarrow Z$ be a lc-trivial fibration.
Then there exists a birational morphism $\mu\colon Z'\rightarrow Z$
such that $M_{Z'}$ is semiample.
\end{conjecture}
We can also state the conjecture for klt-trivial fibrations (cf. ~\cite[Conjecture 7.13.3]{ProkShok}).
\begin{conjecture}\label{bSemiKLT}
Let $f\colon(X,B)\rightarrow Z$ be a klt-trivial fibration.
Then there exists a birational morphism $\mu\colon Z'\rightarrow Z$
such that $M_{Z'}$ is semiample.
\end{conjecture}
Conjecture \ref{bSemiKLT} is true if the fiber of the klt-trivial fibration
is a curve by ~\cite[Theorem 8.1]{ProkShok} and if the fiber is isomorphic to a K3 surface or to an
abelian variety by ~\cite[Theorem 1.2]{Fuji}.
Moreover Conjecture \ref{bSemiLC} is true if $\dim Z=1$ by ~\cite[Theorem 3.5]{Amb05} in the klt case
and by ~\cite[Theorem 1.4]{EFindapp}.

In ~\cite{FG} the authors prove that,
for an lc-trivial fibration, the moduli part is the pull-back of a big divisor.
They consider a center $W$ of $(X,B)$
that is minimal among the centers such that $f(W)=Z$
and they find a quasi-finite morphism $Z''\rightarrow Z$
such that
\begin{enumerate}
\item the restriction $f'\vert_{W'}$ defines a klt-trivial fibration;
\item if we set $M^{min}_{Z'}$ the moduli part of $f'\vert_{W'}$ then we have
$$M^{min}_{Z'}\sim_{\mathbb{Q}}M_{Z'}.$$
\end{enumerate}
By following the same approach as in ~\cite{FG}
we can prove the following result
\begin{teo}\label{mainteo}
Conjecture \ref{bSemiKLT} implies Conjecture \ref{bSemiLC}.
\end{teo}

\section{Definitions and known results}
We will work over $\mathbb{C}$.
In the following $\equiv$, $\sim$ and $\sim_{\mathbb{Q}}$ will respectively indicate
numerical, linear and $\mathbb{Q}$-linear equivalence of divisors.

\begin{defi}
Let $(X,B)$ be a pair and $\nu\colon X'\rightarrow X$
a log resolution of the pair.
We set $$A(X,B)=K_{X'}-\nu^{\ast}(K_X+B)$$
and $$A^{\ast}(X,B)=A(X,B)+\sum_{a(E,X,B)=1} E.$$
\end{defi}

\begin{defi}\label{lctrivialfibr}
A klt-trivial (resp. lc-trivial) fibration $f \colon (X,B) \rightarrow Z$ consists of a 
surjective morphism with connected fibers
of normal varieties $f \colon X \rightarrow Z$ and of a log pair $(X,B)$ satisfying
the following properties:
\begin{enumerate}
\item $(X,B)$ has klt (resp. lc) singularities over the generic
point of $Z$;
\item ${\rm rank}\, f'_{\ast}\,\mathcal{O}_X(\lceil A(X,B)\rceil) = 1$
(resp. ${\rm rank}\, f'_{\ast}\,\mathcal{O}_X(\lceil A^{\ast}(X,B)\rceil) = 1$)
where $f'=f\circ\nu$ and $\nu$ is a given log resolution of the pair $(X,B)$;
\item there exists a positive integer $r$, a rational function $\varphi\in \mathbb{C}(X)$
and a $\mathbb{Q}$-Cartier divisor $D$ on $Z$ such that
$$K_X + B +\frac{1}{r}(\varphi) = f^{\ast}D.$$
\end{enumerate}
\end{defi}

\begin{oss}\label{cartindex}
{\rm
The smallest possible $r$ that can appear in Definition \ref{lctrivialfibr} is the minimum of the set
$$\{m\in\mathbb{N}|\,m(K_X+B)|_F\sim 0\}$$
that is the Cartier index of the fiber.
We will always assume that the $r$ that appears in the formula is the smallest one.
}
\end{oss}

\begin{defi}
Let $P\subseteq Z$ be a prime Weil divisor.
The log canonical threshold $\gamma_P$ of $f^{\ast}(P)$ with respect to the pair $(X,B)$ is defined as follows.
Let $\bar{Z}\rightarrow Z$ be a resolution of $Z$.
Let $\mu\colon\bar{X}\rightarrow X$ be the birational morphism obtained as a desingularisation of the main component
of $X\times_Z\bar{Z}$. Let $\bar{f}\colon\bar{X}\rightarrow\bar{Z}$.
Let $\bar{B}$ be the divisor defined by the relation
$$K_{\bar{X}}+\bar{B}=\mu^{\ast}(K_X+B).$$
Let $\bar{P}$ be the strict transform of $P$ in $\bar{Z}$.
Set
$$\gamma_P=\sup\{t\in\mathbb{Q}|\,(\bar{X},\bar{B}+t\bar{f}^{\ast}(\bar{P})) {\rm \:is\: lc\: over\:} \bar{P}\}.$$
We define the {\rm discriminant} of $f \colon (X,B) \rightarrow Z$ as
\begin{eqnarray}\label{discriminant}
B_Z&=&\sum_{P}(1-\gamma_P)P.
\end{eqnarray}
\end{defi}
\begin{oss}[~\cite{Kaw2}, p.14 ~\cite{Amb99}]\label{fibraSNC}
{\rm
The log canonical threshold $\gamma_P$ is a rational number and the sum above is finite,
thus the discriminant is a $\mathbb{Q}$-Weil divisor
}
\end{oss}
\begin{defi}
Fix $\varphi\in\mathbb{C}(X)$ such that $K_X + B +\frac{1}{r}(\varphi) = f^{\ast}D$.
Then there exists a unique divisor $M_Z$ such that we have
\begin{eqnarray}\label{cbf}
K_X + B +\frac{1}{r}(\varphi) &=& f^{\ast}(K_Z+B_Z+M_Z)
\end{eqnarray}
where $B_Z$ is as in (\ref{discriminant}).
The $\mathbb{Q}$-Weil divisor $M_Z$ is called the {\rm moduli part}.
\end{defi}
We have the two following results.\\
\begin{teo}[Theorem 0.2 ~\cite{Amb04}, ~\cite{Corti}]\label{nefness}
Let $f \colon(X,B)\rightarrow Z$ be an lc-trivial fibration.
Then there exists a proper birational morphism
$Z'\rightarrow Z$ with the following properties:
\begin{description}
\item[(i)] $K_{Z'}+B_{Z'}$ is a $\mathbb{Q}$-Cartier divisor, 
and for every proper birational morphism $\nu \colon Z''\rightarrow Z'$
$$\nu^{\ast}(K_{Z'}+B_{Z'}) = K_{Z''}+B_{Z''}.$$
\item[(ii)] $M_{Z'}$ is a nef $\mathbb{Q}$-Cartier divisor and 
for every proper birational morphism $\nu \colon Z''\rightarrow Z'$
$$\nu^{\ast}(M_{Z'}) = M_{Z''}.$$
\end{description}
\end{teo}
\begin{prop}[Proposition 5.5 ~\cite{Amb04}]\label{basechange}
Let $f\colon (X,B)\rightarrow Z$ be an lc-trivial fibration.
Let $\tau\colon Z'\rightarrow Z$ be a generically finite projective
morphism from a non-singular variety $Z'$. Assume there exists a simple
normal crossing divisor $\Sigma_{Z'}$ on $Z'$ which contains $\tau^{-1}\Sigma_Z$ and the
locus where $\tau$ is not \'etale. Let $M_{Z'}$ be the moduli part of the induced
lc-trivial fibration $f'\colon (X',B')\rightarrow Z'$.
Then $M_{Z'}=\tau^{\ast}M_Z$.
\end{prop}
The formula (\ref{cbf}) is called the \textit{canonical bundle formula}.

The following result is a well known result.
We recall the proof here for the reader's convenience.
\begin{lem}\label{geq2}
Let $\tau\colon W\rightarrow Y$ be a generically finite morphism.
Let $D$ be a divisor on $Y$ and assume that $\tau^{\ast}D$ is base point free.
Let $d=\deg\tau$.
Then $${\rm Codim}({\rm Bs}|dD|)\geq 2.$$
\end{lem}
\begin{proof}
Consider the Stein factorization of $\tau$
$$
\xymatrix{
W\ar[rr]^{\tau}\ar[rd]_{g}& &Y\\
 &Y'.\ar[ru]_{h}&
}
$$
The morphism $g$ is birational.
Let $z\in Y\backslash \tau({\rm Exc}(g))$.
Then the set $\tau^{-1}(z)$ is finite.
A general element $E\in|\tau^{\ast}D|$
is such that ${\rm Supp}E\cap\tau^{-1}(z)=\emptyset$,
hence $z\not\in{\rm Supp}\tau_{\ast}E$.
Since $\tau_{\ast}E\in|dD|$, we proved that
$${\rm Bs}|dD|\subseteq \tau({\rm Exc}(g))$$
which is a subset of codimension at least 2.
\end{proof}
\section{Proof of the main result}
This section is devoted to the proof of Theorem \ref{mainteo}.
We will be dealing with lc-trivial fibrations that are not klt-trivial fibrations.
The main tool is the subadjuction formula for centers that are minimal over the generic point of $Z$.
The next remark explains the geometric construction used in ~\cite{FG} and the notation
that will be used. 

\begin{oss}[cf. ~\cite{FG}]\label{constr}
{\rm
Let $f\colon(X,B)\rightarrow Z$ be an lc-trivial fibration such that $(X,B)$ is not klt over the generic point of $Z$.
Then there exists a center $W\subseteq X$ that is minimal among the centers $C$ of $(X,B)$
such that $f(C)=Z$.
Let $$
\xymatrix{
 W\ar[rd]_{\alpha}\ar[rr]^{f\vert_W}&& Z\\
 &Z_1\ar[ru]_{\beta}&
}
$$ be the Stein factorization of $f\vert_W$. Let $Z_1'$ be a desingularisation of $Z_1$.
Consider the lc-trivial fibration induced by the base change
$$
\xymatrix{
(X'_1,B'_1)\ar[d]_{f'_1}\ar[r]&(X,B)\ar[d]_{f}\\
Z'_1\ar[r]&Z.
}
$$
Let $W'_1$ be the preimage of $W$ in $X'_1$ and let $n\colon W'^n_1\rightarrow W'_1$
be its normalization.
The variety $W'^n_1$ is birational to $W$ and, by adjunction (see ~\cite[Proposition 3.9.2]{Corti}),
there exists $B^{W_1}$ such that
\begin{itemize}
\item the pair $(W'^n_1,B^{W_1})$ is klt;
\item we have $K_{W'^n_1}+B^{W_1}=n^{\ast}(K_{X'_1}+B)$.
\end{itemize}
Then the restriction of $f'_1$ induces a klt-trivial fibration
$$g_1=f'_1\circ n\colon (W'_1,B^{W_1})\rightarrow Z'_1$$
to whom we can associate a discriminant $B_{Z'_1}^{min}$
and a moduli part $M_{Z'_1}^{min}$.
If $B_{Z'_1}$ and $M_{Z'_1}$ are discriminant and moduli part of $f'_1$,
we have $$K_{Z'_1}+B_{Z'_1}+M_{Z'_1}\sim_{\mathbb{Q}}K_{Z'_1}+B^{min}_{Z'_1}+M^{min}_{Z'_1}.$$
}
\end{oss}
Though not explicitely stated, the following proposition is the key result in ~\cite{FG}.
\begin{teo}[Theorem 1.1 ~\cite{FG}]\label{fujigongyo}
Let $f\colon(X,B)\rightarrow Z$ be an lc-trivial fibration.
Assume that $(X,B)$ is not klt over the generic point of $Z$
and let $W$ be center of $(X,B)$
that is minimal over the generic point of $Z$ and such that $f(W)=Z$.\\
Assume that $M_Z$ is such that for any birational morphism $\mu\colon Z_1\rightarrow Z$
we have $M_{Z_1}=\mu^{\ast}M_Z$.\\
Then there exists a generically finite mophism $\tau\colon Z'\rightarrow Z$ with the following properties.
Let $f'\colon (X',B')\rightarrow Z'$ be the lc-trivial fibration
obtained by base change, let $W'$ be the preimage of $W$ in $X'$ and
let $n\colon W'^n\rightarrow W'$ be its normalization. Then
\begin{enumerate}
\item $g=f'\vert_{W'}\circ n\colon (W'^n,B^W)\rightarrow Z'$ is a klt-trivial fibration;
\item for any birational morphism $\mu\colon Z''\rightarrow Z'$ we have $M_{Z''}^{min}=\mu^{\ast}M_{Z'}^{min}$;
\item $B_{Z'}^{min}=B_{Z'}$.
\end{enumerate}
\end{teo}

Let us prove Theorem \ref{mainteo}.
\begin{proof}[proof of Theorem \ref{mainteo}]
Let $f\colon(X,B)\rightarrow Z$ be an lc-trivial fibration.
If $f$ is a klt-trivial fibration, we are done.
Then we can assume that $(X,B)$ is not klt over the generic point of $Z$.
Let $W$ be an arbitrary log
canonical center of $(X, B)$ that is dominant onto $Z$ and is minimal
over the generic point of $Z$.
By Theorem \ref{nefness}, we can assume that $M_Z$ is nef and that for any
birational morphism $\mu\colon Z_1\rightarrow Z$
we have $M_{Z_1}=\mu^{\ast}M_Z$.
Let $\tau\colon Z'\rightarrow Z$ be a generically finite base change
as in Proposition \ref{fujigongyo}.
Then, by using the same notation as in Remark \ref{constr}, we have $M^{min}_{Z'}\sim_{\mathbb{Q}}M_Z$.
By hypothesis there exists a birational morphism $\nu\colon Z''\rightarrow Z'$
such that $M^{min}_{Z''}$ is semiample.
Let $m$ be an integer such that 
\begin{itemize}
\item $mM^{min}_{Z''}$ is base point free;
\item $mM^{min}_{Z'}\sim mM_{Z'}$.
\end{itemize}
Set $\tilde{\tau}=\tau\circ\nu$. The morphism $\tilde{\tau}$ is generically finite
and let $d=\deg \tilde{\tau}$.
By Lemma \ref{geq2} $${\rm Codim}({\rm Bs}|mdM_Z|)\geq 2.$$
Let $\mu\colon\bar{Z}\rightarrow Z$ be a resolution of the base locus of $|mdM_Z|$.
Then $$|\mu^{\ast}(mdM_Z)|=|M|+F,$$
the linear system $|M|$ is base point free and $F$ is the base locus of $|\mu^{\ast}(mdM_Z)|$.
Let $\bar{Z}'$ be a desingularisation of the main component of the fiber product $Z''\times_{Z}\bar{Z}$.
We have the following commutative diagram:
$$
\xymatrix{
Z'' \ar[r]^{\tilde{\tau}} & Z\\
\bar{Z}''\ar[u]^{\mu''}\ar[r]_{\sigma} & \bar{Z}\ar[u]_{\mu}
}
$$
Since the morphism $\mu''$ is birational, we have $M_{\bar{Z}''}=\mu''^{\ast}M_{Z''}$
and $M_{\bar{Z}''}$ is semiample.
Moreover the degree of $\sigma$ is the same as the degree of $\tilde{\tau}$
and $M_{\bar{Z}''}=\sigma^{\ast}M_{\bar{Z}}$.
Then, by Lemma \ref{geq2}, we have $${\rm Codim}({\rm Bs}|mdM_{\bar{Z}}|)\geq 2.$$
On the other hand, since $$|mdM_{\bar{Z}}|=\mu^{\ast}|mdM_Z|,$$
the base locus of $mdM_{\bar{Z}}$ has pure codimension 1.
This implies that $F=0$ and that the linear system $|mdM_Z|$ is base point free.
\end{proof}

\end{document}